\documentclass[submission]{FPSAC2026}

\usepackage{tikz-cd}
\usepackage{bm, relsize}
\usepackage{comment}

\newtheorem{theorem}{Theorem}

\newtheorem{defn}{Definition}
\newtheorem{lem}{Lemma}
\newtheorem{proposition}{Proposition}

\newtheorem{corollary}{Corollary}
\newtheorem{example}{Example}

\newtheorem{remark}{Remark}

\newcommand{\w}{\omega}

\newcommand{\QQ}{\mathbb{Q}} 
\newcommand{\ZZ}{\mathbb{Z}} 
\newcommand{\id}{\operatorname{id}} 
\newcommand{\kk}{\mathbf{k}} 
\newcommand{\BB}{\mathbf{B}} 

\newcommand{\Des}{\operatorname{Des}}

\newcommand{\Comp}{\operatorname{Comp}}
\newcommand{\Par}{\operatorname{Par}}
\newcommand{\Set}{\operatorname{Set}}
\newcommand{\cLRM}{\operatorname{cLRM}}
\newcommand{\LRM}{\operatorname{LRM}}
\newcommand{\lex}{\operatorname{lex}}
\newcommand{\set}[1]{\left\{ #1 \right\}}
\newcommand{\abs}[1]{\left| #1 \right|}
\newcommand{\tup}[1]{\left( #1 \right)}
\newcommand{\ive}[1]{\left[ #1 \right]}
\newcommand{\ul}{\underline}

\makeatletter
\providecommand*{\shuffle}{%
  \mathbin{\mathpalette\shuffle@{}}%
}
\newcommand*{\shuffle@}[2]{%
  \sbox0{$#1\vcenter{}$}%
  \kern .15\ht0 
  \rlap{\vrule height .25\ht0 depth 0pt width 2.5\ht0}%
  \raise.1\ht0\hbox to 2.5\ht0{%
    \vrule height 1.75\ht0 depth -.1\ht0 width .17\ht0 %
    \hfill
    \vrule height 1.75\ht0 depth -.1\ht0 width .17\ht0 %
    \hfill
    \vrule height 1.75\ht0 depth -.1\ht0 width .17\ht0 %
  }%
  \kern .15\ht0 
}
\makeatother

\title[The left-to-right minima basis]{The left-to-right minima basis of the group algebra of the symmetric group (updated version)}

\author[D. Grinberg \and E.A. Vassilieva]{Darij Grinberg\thanks{\href{mailto:darijgrinberg@gmail.com}{darijgrinberg@gmail.com}}\addressmark{1}, \and Ekaterina A. Vassilieva\thanks{\href{mailto:katya@lix.polytechnique.fr}{katya@lix.polytechnique.fr}}\addressmark{2}}

\address{\addressmark{1}Department of Mathematics, Drexel University, Philadelphia, PA 19104, USA \\ \addressmark{2}LIX, Ecole Polytechnique, Palaiseau, France}

\received{\today}

\abstract{
We introduce a new basis of the group algebra of the symmetric group, built using the left-to-right minima sets of permutations.
We show that on this basis, the descent algebra acts by triangular operators, thus making it an analogue of a cellular basis.
The proof involves Dynkin elements (nested commutators) of the free algebra and their interactions with the $\BB$-basis.
}

\resume{
Nous introduisons une nouvelle base de l’algèbre de groupe du groupe symétrique, construite à partir des ensembles de minima gauches des permutations. Nous montrons que, sur cette base, l'algèbre de descente agit par opérateurs triangulaires, ce qui en fait un analogue d'une base cellulaire.
La démonstration fait intervenir les éléments de Dynkin (commutateurs imbriqués) de l'algèbre libre et leurs interactions avec la base $\BB$.
}


\usepackage[backend=bibtex]{biblatex}
\begin{document}

\maketitle

\section{Introduction}

For a nonnegative integer \(n\), let \(S_n\) denote the symmetric group on
\([n]=\{1,2,\dots,n\}\).  Fix a ground ring \(\kk\) and let
\(\mathcal{A} = \kk[S_n]\) be the group algebra of \(S_n\) over
\(\kk\). For \(w\in S_n\), the \emph{descent set} of \(w\) is
\[
\Des (w) =\{\, i \in \ive{n-1} \mid w(i)>w(i+1)\,\}.
\]
For every subset \(I\subseteq[n-1]\) define the element
\[
\mathbf{B}_I=\sum_{\substack{w\in S_n;\\ \Des (w) \subseteq I}} w
\in\mathcal{A}.
\]
The span of the elements \(\mathbf{B}_I\) is the classical \emph{descent
algebra} \(\Sigma_n\subseteq\mathcal{A}\) introduced by Solomon in \cite{Sol76}. We call the family
\((\mathbf{B}_I)_{I\subseteq[n-1]}\) the
\emph{\(\mathbf{B}\)-basis} of \(\Sigma_n\).
The aim of this paper is to construct and study a new basis
\((\beta_w)_{w\in S_n}\) of the whole algebra \(\mathcal{A}\) which is
compatible with the right-action of the descent algebra in the following
sense: there should exist a partial order \(\preceq\) on \(S_n\) such that
for every \(w\in S_n\) and every \(I\subseteq[n-1]\) one has
\[
\mathbf{B}_I\cdot \beta_w\;\in\;
\operatorname{span}\{\,\beta_u\mid u\preceq w\,\}.
\]
In other words, the action of \(\Sigma_n\) on \(\mathcal A\) should
be represented by triangular matrices on this basis.
Our basis will furthermore be compatible with a
natural family of right ideals of \(\mathcal A\) generated by the
\(\mathbf{B}\)-basis.  In the sequel we introduce a
poset-indexed filtration of \(\mathcal A\) built from the right ideals
\(\mathbf{B}_\alpha\mathcal A\) (reindexing \(\mathbf{B}_I\) by
compositions), construct our candidate basis, and prove that it is a
basis adapted to the filtration.

The paper is organised as follows.  In Section~\ref{sec.filtration} we
define the filtration, recall the necessary combinatorics of integer
compositions, and record basic properties of the \(\mathbf{B}\)-basis
action. Section~\ref{sec.lrm-basis} introduces the left-to-right-minima basis
and states the main theorem connecting it with the filtration.
The remaining sections are devoted to proofs and examples.

\section{A poset-indexed filtration of the group algebra}
\label{sec.filtration}

\subsection{Compositions and refinement}
\label{sec:compositions}

A \emph{composition} \(\alpha=(\alpha_1,\dots,\alpha_p)\) of a nonnegative
integer \(n\) (notation \(\alpha\vDash n\)) is a finite sequence of positive
integers with \(|\alpha|:=\sum_i\alpha_i=n\); write \(\ell(\alpha)=p\) for
its length.
Let $\widetilde{\alpha}$ denote the \emph{underlying partition} of $\alpha$, that is, the partition obtained by rearranging the parts of $\alpha$ in weakly decreasing order.
Two compositions $\alpha$ and $\beta$ will be called \emph{anagrams}
(of each other) if they satisfy
$\widetilde{\alpha} =\widetilde{\beta}$.
Let \(\Comp_n\) be the set of compositions of \(n\).
A well-known bijection between compositions of \(n\) and subsets of
\([n-1]\) sends each composition \(\alpha=(\alpha_1,\dots,\alpha_p)
\in \Comp_n\) to
\[
\operatorname{Set}(\alpha)
 := \{\alpha_1,\ \alpha_1+\alpha_2,\ \dots,\ \alpha_1+\cdots+\alpha_{p-1}\}\subseteq[n-1],
\]
and its inverse map sends each subset \(I=\{i_1<\cdots<i_{p-1}\}\subseteq[n-1]\) to
\(\operatorname{Comp}(I)=(i_1-i_0,i_2-i_1,\dots,i_p-i_{p-1})\in\Comp_n\),
where we set \(i_0 := 0\) and \(i_p := n\).
Accordingly we reindex the \(\mathbf{B}\)-basis by compositions by
setting
\(\mathbf{B}_\alpha:=\mathbf{B}_{\operatorname{Set}(\alpha)}\) for
\(\alpha\in\Comp_n\).

Compositions are partially ordered by refinement:

\begin{defn}[refinement]
Let \(\alpha=(\alpha_1,\dots,\alpha_m)\) and
\(\beta=(\beta_1,\dots,\beta_p)\) be compositions of the same integer
\(n\). We say that \(\alpha\) \emph{refines} \(\beta\), written
\(\alpha\preceq\beta\), if \(\alpha\) can be partitioned into \(p\)
contiguous sub-compositions \(\alpha^{(1)},\dots,\alpha^{(p)}\) with
\(|\alpha^{(j)}|=\beta_j\) for all \(j\).
\end{defn}

Partitions are also partially ordered by refinement, but of a different kind:

\begin{defn}[partition refinement]
Let \(\lambda,\mu\vdash n\) be partitions of \(n\).  We say that \(\lambda\) is a
\emph{partition refinement} of \(\mu\), written \(\lambda\preceq_\pi\mu\),
if some anagram of \(\lambda\) refines \(\mu\)
(as a composition).
\end{defn}

Equivalently, this relation $\preceq_\pi$ can be defined as follows:
Two partitions $\lambda = \tup{\lambda_1, \dots, \lambda_k}$
and $\mu = \tup{\mu_1, \dots, \mu_l}$ satisfy $\lambda \preceq_\pi \mu$
if and only if there exists a map $f : [k] \to [l]$ such that
each $j \in [l]$ satisfies $\mu_j = \sum \lambda_i$, where the sum
ranges over all $i \in f^{-1}(j)$.

Both relations $\preceq$ and $\preceq_\pi$ are weak partial orders;
their strict versions will be called $\prec$ and $\prec_\pi$.
Note that $\alpha \preceq \beta$ implies
$\widetilde{\alpha} \preceq_\pi \widetilde{\beta}$.

For any composition $\alpha = \tup{\alpha_1, \ldots, \alpha_k} \vDash n$,
we define the \emph{\(\alpha\)-partial sums} to be the $k+1$ integers
\[
\alpha_{\le i} := \alpha_1 + \alpha_2 + \cdots + \alpha_i
\qquad \text{for all } 0 \leq i \leq k,
\]
and we define the \emph{\(\alpha\)-blocks} to be the $k$ subsets%
\footnote{The notation $\ive{u, v}$ for two integers $u$ and $v$
denotes the integer interval $\set{x \in \ZZ \mid u \leq x \leq v}$.}
\begin{align}
\Set(\alpha)_i := \ive{ \alpha_{\le i-1} + 1, \ \alpha_{\le i} }
\text{ of } \ive{n}
\qquad \text{for all } i \in [k].
\label{eq.alpha-blocks}
\end{align}
These \(\alpha\)-blocks
\(\Set(\alpha)_1, \Set(\alpha)_2, \ldots, \Set(\alpha)_k\)
are $k$ disjoint nonempty integer intervals
of respective lengths \(\alpha_1, \alpha_2, \ldots, \alpha_k\),
and their union is $[n]$.
The $\BB$-basis element $\BB_\alpha$ is then the sum of all
permutations $w \in S_n$ that increase on each $\alpha$-block.

\subsection{Action of the \(\mathbf{B}\)-basis and right ideals}
\label{sec:B-action}

Elements of the \(\mathbf{B}\)-basis can be conveniently
multiplied using the
\emph{Solomon's Mackey formula} (see e.g. \cite[Theorem 2]{Willigenburg98}).
If \(M\) is a 
nonnegative integer matrix
(i.e., matrix with nonnegative integer entries),
denote by \(\operatorname{rows}(M)\) and
\(\operatorname{cols}(M)\) its row- and column-sum vectors and by
\(\operatorname{read}(M)\) the composition obtained by concatenating the
rows of \(M\) (omitting zeros).  Given \(\alpha,\beta\in\Comp_n\), write
\(\mathbb N^{\alpha,\beta}\) for the set of nonnegative integer matrices
with column sums \(\alpha\) and row sums \(\beta\).  Solomon's formula
(in our notation) reads
\begin{align}
\mathbf{B}_\alpha\mathbf{B}_\beta
   = \sum_{M\in\mathbb N^{\alpha,\beta}} \mathbf{B}_{\operatorname{read}(M)}.
\label{eq.solmac}
\end{align}
In particular every composition \(\operatorname{read}(M)\) on the right
refines \(\beta\) and satisfies \(\widetilde{\operatorname{read}(M)}
\preceq_\pi \widetilde{\alpha}\); consequently
left multiplication by \(\mathbf{B}_\alpha\) is triangular on the basis
\((\mathbf{B}_\gamma)_{\gamma\in\Comp_n}\) of $\Sigma_n$ with respect to refinement.

\begin{defn}[$\mathbf{B}$-right ideals]
For \(\beta\in\Comp_n\) define the right ideal
\[
\mathcal R_\beta := \mathbf{B}_\beta\,\mathcal A \subseteq \mathcal A.
\]
\end{defn}

It is easy to see that the right ideal \(\mathcal R_\alpha\) depends only
on the multiset of parts of \(\alpha\) (i.e.\ on the corresponding
integer partition), not on their order.

\begin{lem}\label{prop.Rbeta.sym}
If \(\alpha,\beta\in\Comp_n\) are anagrams
(that is, \(\widetilde{\alpha} = \widetilde{\beta}\)),
then \(\mathcal R_\alpha=\mathcal R_\beta\).
\end{lem}

\begin{proof}
If \(\beta\) is an anagram of \(\alpha\), then there
exists a permutation \(w\in S_n\) sending the \(\alpha\)-blocks to the
\(\beta\)-blocks while preserving the order of the elements in each
block.
Thus, we easily find
\( \BB_\beta w = \BB_\alpha \), and therefore
\(\mathbf{B}_\alpha\mathcal A=\mathbf{B}_\beta w\mathcal A=\mathbf{B}_\beta\mathcal A\).
\end{proof}

The family \((\mathcal R_\beta)_{\beta\in\Comp_n}\) is a poset-indexed
filtration of \(\mathcal A\) (ordered by refinement).

\begin{lem}\label{prop.Rbeta.sub}
If \(\beta,\gamma\in\Comp_n\) and \(\beta\preceq\gamma\) (i.e. \(\beta\) refines \(\gamma\)), then
\(\mathcal R_\beta\subseteq\mathcal R_\gamma\).
\end{lem}

\begin{proof}
Write the composition $\gamma$ as $\gamma=\left(  \gamma_{1},\gamma_{2}%
,\ldots,\gamma_{k}\right)  $, and let
\[
\left(  I_{1},I_{2},\ldots
,I_{k}\right)  :=\left(  \operatorname*{Set}\left(  \gamma\right)
_{1},\operatorname*{Set}\left(  \gamma\right)  _{2},\ldots,\operatorname*{Set}%
\left(  \gamma\right)  _{k}\right)
\]
be the list of the $\gamma$-blocks. Let
$\mathbf{B}_{\gamma\rightarrow\beta}$ be the sum of those permutations $w\in
S_{n}$ that preserve the $\gamma$-blocks (i.e., satisfy $w\left(
I_{p}\right)  =I_{p}$ for all $p$) and satisfy $\Des w
\subseteq\operatorname*{Set}\left(  \beta\right)  $. By the definition of
$\mathbf{B}_{\beta}$, we have%
\begin{align*}
\mathbf{B}_{\beta} &  =\sum_{\substack{w\in S_{n};\\ \Des w
\subseteq\operatorname*{Set}\left(  \beta\right)  }}w=\sum
_{\substack{\mathbf{J}=\left(  J_{1},J_{2},\ldots,J_{k}\right)  \text{
is}\\\text{a set composition of }\left[  n\right]  ;\\\left\vert
J_{p}\right\vert =\gamma_{p}\text{ for each }p}}\ \ \underbrace{\sum
_{\substack{w\in S_{n};\\ \Des w \subseteq\operatorname*{Set}%
\left(  \beta\right)  ;\\w\left(  I_{p}\right)  =J_{p}\text{ for each }p}%
}w}_{\substack{=\rho_{\mathbf{J}}\mathbf{B}_{\gamma\rightarrow\beta}%
\text{,}\\\text{where }\rho_{\mathbf{J}}\text{ is the unique permutation of
}\left[  n\right]  \\\text{that sends each }I_{p}\text{ to }J_{p}\text{
order-preservingly}}}\\
&  =\underbrace{\sum_{\substack{\mathbf{J}=\left(  J_{1},J_{2},\ldots
,J_{k}\right)  \text{ is}\\\text{a set composition of }\left[  n\right]
;\\\left\vert J_{p}\right\vert =\gamma_{p}\text{ for each }p}}\rho
_{\mathbf{J}}}_{\substack{=\mathbf{B}_{\gamma}\\\text{(since the addends }%
\rho_{\mathbf{J}}\text{ of this sum are}\\\text{precisely the }w\in
S_{n}\text{ satisfying } \Des w \subseteq\operatorname*{Set}%
\left(  \gamma\right)  \text{)}}}\mathbf{B}_{\gamma\rightarrow\beta
}=\mathbf{B}_{\gamma}\mathbf{B}_{\gamma\rightarrow\beta}.
\end{align*}
Hence
\(\mathbf{B}_\beta\mathcal A=\mathbf{B}_\gamma\mathbf{B}_{\gamma\to\beta}\mathcal A\subseteq\mathbf{B}_\gamma\mathcal A\).
In other words,
\( \mathcal R_\beta \subseteq \mathcal R_\gamma \).
\end{proof}

\begin{proposition}\label{cor.Rbeta.subpar}
If \(\beta,\gamma\in\Comp_n\) and \(\widetilde{\beta}\preceq_\pi\widetilde{\gamma}\), then
\(\mathcal R_\beta\subseteq\mathcal R_\gamma\).
\end{proposition}

\begin{proof}
Lemma~\ref{prop.Rbeta.sym} shows that
\(\mathcal R_\beta\) does not change if we permute
the entries of \(\beta\).
But this allows us to transform \(\widetilde{\beta}\preceq_\pi\widetilde{\gamma}\) into
\(\beta\preceq\gamma\),
and then the claim follows from Lemma~\ref{prop.Rbeta.sub}.
\end{proof}


Hence \((\mathcal R_\beta)_{\beta\in\Comp_n}\) is a filtration of
\(\mathcal A\) by right
\(\mathcal A\)-submodules that can also be relabelled (removing
redundant copies) as \((\mathcal R_\beta)_{\beta\in\Par_n}\),
where \(\Par_n\) is the set of all partitions of \(n\).
Furthermore, these right ideals are stable under the
left action of \(\Sigma_n\):

\begin{proposition}\label{prop.Rbeta.subbi}
For every \(\beta\in\Comp_n\), the subspace \(\mathcal R_\beta\) is a
\((\Sigma_n,\mathcal A)\)-subbimodule of \(\mathcal A\).
\end{proposition}

\begin{proof}
Right \(\mathcal A\)-stability is immediate from the definition.
Left-stability follows from Solomon's formula \eqref{eq.solmac}: for any
\(\alpha\in\Comp_n\) the product \(\mathbf{B}_\alpha\mathbf{B}_\beta\) is a
linear combination of \(\mathbf{B}_\gamma\) with \(\gamma\) refining
\(\beta\); by Lemma~\ref{prop.Rbeta.sub} each such \(\mathbf{B}_\gamma\)
generates a right ideal contained in \(\mathcal R_\beta\).  Therefore
\(\mathbf{B}_\alpha\mathcal R_\beta\subseteq\mathcal R_\beta\) for all
\(\alpha\), whence \(\Sigma_n\cdot\mathcal R_\beta\subseteq\mathcal R_\beta\).
\end{proof}
We look at the quotient spaces
\[
\mathcal R_\beta/\sum_{\alpha\prec\beta}\mathcal R_\alpha,
\]
where the sum runs over compositions \(\alpha\) strictly finer than
\(\beta\), or, equivalently, over compositions \(\alpha\) satisfying
\(\widetilde{\alpha} \prec_\pi \widetilde{\beta}\)
(the latter sum has more addends than the former, but
Lemma~\ref{prop.Rbeta.sym} shows that these extra addends
merely duplicate existing addends).
Left multiplication by \(\mathbf{B}_\alpha\) acts on these
quotient spaces by scalar multiplication:

\begin{theorem}\label{thm.Rbeta.act-on-quot}
Fix \(\alpha,\beta\in\Comp_n\). Let \(\eta_\beta(\alpha)\) be the number
of functions \(f:[\ell(\beta)]\to[\ell(\alpha)]\) such that for every
\(j\in[\ell(\alpha)]\) one has \(\alpha_j=\sum_{i\in f^{-1}(j)}\beta_i\).
Then left multiplication by \(\mathbf{B}_\alpha\) acts on the quotient
\(\mathcal R_\beta/\sum_{\alpha'\prec\beta}\mathcal R_{\alpha'}\) as
multiplication by the scalar \(\eta_\beta(\alpha)\).
\end{theorem}

\begin{proof}
This is an easy corollary of Solomon's formula \eqref{eq.solmac}.
Indeed, all matrices $M$ on the right hand side of
\eqref{eq.solmac} satisfy $\operatorname{read}(M) \preceq \beta$.
Among these matrices $M$, the ones that satisfy
$\operatorname{read}(M) = \beta$ are precisely the ones that have exactly one
nonzero entry on each row.
The number of these latter matrices is
\(\eta_\beta(\alpha)\), since they can be encoded by the
function $f : \ive{\ell(\beta)} \to \ive{\ell(\alpha)}$ where
$f\tup{i}$ is the position of the unique nonzero entry in the
$i$-th row of $M$.
Hence, \eqref{eq.solmac} becomes
\begin{align}
\BB_\alpha \BB_\beta = \eta_\beta(\alpha) \BB_\beta +
\tup{\text{a sum of } \BB_{\alpha'} \text{ with } \alpha' \prec \beta}.
\label{eq.BaBb=sum-prec}
\end{align}
Since $\mathcal{R}_\beta = \BB_\beta \mathcal{A}$, this proves the claim.
\end{proof}

\begin{remark}
It can be shown that if $\kk$ is a $\QQ$-algebra, then
Proposition~\ref{cor.Rbeta.subpar} can be strengthened
as follows:
If \(\beta,\gamma\in\Comp_n\) and \(\widetilde{\beta}\preceq_\pi\widetilde{\gamma}\), then
\(\BB_\beta \Sigma_n \subseteq \BB_\gamma \Sigma_n \).
However, our Proposition~\ref{cor.Rbeta.subpar} holds
independently of $\kk$.
\end{remark}

\section{The left-to-right minima basis of $\mathcal{A}$}
\label{sec.lrm-basis}
\subsection{Left-to-right minima of a permutation}

We begin by recalling a classical statistic on permutations.

\begin{defn}[Left-to-right minima]
Let $w\in S_n$.  
The \emph{left-to-right minima} (short: \emph{LRMs})
of $w$ are the entries 
\[
\operatorname{LRM}(w)
 := \Bigl\{\, i\in[n]\ \big|\ \forall\, k<w^{-1}(i),\;\; w(k)>i \Bigr\}.
\]
Thus $i$ belongs to $\operatorname{LRM}(w)$ if it is smaller than every entry appearing to its left in the one-line notation of $w$.
Clearly, $\operatorname{LRM}(w)$ always contains $1$ and $w(1)$.
We also define the shifted set
\[
\operatorname{LRM}'(w)
 := \{\, \ell-1\mid \ell\in\operatorname{LRM}(w),\;\ell>1\}
 \subseteq [n-1],
\]
and finally the associated composition
(see \S\ref{sec:compositions} for notations)
\[
\operatorname{cLRM}'(w)
 := \Comp\bigl(\operatorname{LRM}'(w)\bigr).
\]
\end{defn}

The following easy fact relates the LRMs of a permutation to those of its inverse.

\begin{remark}
\label{lem.LRM.inverse}
For every $w\in S_n$, one has
$\operatorname{LRM}(w)= w\bigl(\operatorname{LRM}(w^{-1})\bigr)$.
\end{remark}


\begin{example}
Lemma~\ref{lem.LRM.inverse} says that the LRMs of $w^{-1}$ occur at the \emph{LRM positions} of $w$.  
For example, if $w = \mathbf{6}7\mathbf{2}49\mathbf{1}853$,
then $\operatorname{LRM}(w)=\{6,2,1\}$,
whereas
$w^{-1} = \mathbf{63}948\mathbf{1}275$
and $\operatorname{LRM}(w^{-1})=\{6,3,1\}$,
which matches the fact that the LRMs of $w$ appear at positions $6,3,1$.
\end{example}

\subsection{A new basis of the group algebra}

For each permutation $w\in S_n$ we consider the element
\[
\mathbf{B}_{\operatorname{LRM}'(w)}\, w \;\in\; \mathcal{A}.
\]

We shall prove that these vectors form a basis of $\mathcal{A}$.
Two lemmas will guide our way.

\begin{lem}
\label{lem.lastdes}Let $u\in S_{n}$ be a permutation that has at least one
descent. Let $m$ be the largest descent of $u$. Let $\ell =m+1$. Then:

\begin{enumerate}
\item[\textbf{(a)}] We have $u\left( \ell \right) <\ell $.

\item[\textbf{(b)}] We have $u\left( i\right) \leq i$ for all $i>\ell $.
\end{enumerate}
\end{lem}

\begin{proof}
\textbf{(a)} Since $\ell -1=m$ is the largest descent of $u$, we have%
\begin{equation}
u\left( \ell -1\right) >u\left( \ell \right) <u\left( \ell +1\right)
<u\left( \ell +2\right) <\cdots <u\left( n\right) .  \label{pf.lem.lastdes.1}
\end{equation}%
Thus, the $n-\ell +1$ distinct numbers $u\left( \ell -1\right) ,\ u\left(
\ell +1\right) ,\ u\left( \ell +2\right) ,\ \ldots ,\ u\left( n\right) $ are
all larger than $u\left( \ell \right) $, hence belong to the interval $\left[
u\left( \ell \right) +1,\ n\right] $. Therefore, this interval must contain
at least $n-\ell +1$ numbers. In other words, $\left\vert \left[ u\left(
\ell \right) +1,\ n\right] \right\vert \geq n-\ell +1>n-\ell $. Since $%
\left\vert \left[ u\left( \ell \right) +1,\ n\right] \right\vert =n-u\left(
\ell \right) $, this rewrites as $n-u\left( \ell \right) >n-\ell $, whence $%
u(\ell )<\ell $. This proves part \textbf{(a)}. 
\medskip

\textbf{(b)} Let $i>\ell $. Then, 
\eqref{pf.lem.lastdes.1} yields $u(i)<u(i+1)<\cdots <u(n)$. Thus, the $n-i$
distinct numbers $u\left( i+1\right) ,\ u\left( i+2\right) ,\ \ldots ,\
u\left( n\right) $ are all larger than $u\left( i\right) $, hence belong to
the interval $\left[ u\left( i\right) +1,\ n\right] $. Therefore, this
interval must contain at least $n-i$ numbers. In other words, $\left\vert %
\left[ u\left( i\right) +1,\ n\right] \right\vert \geq n-i$. Since $%
\left\vert \left[ u\left( i\right) +1,\ n\right] \right\vert =n-u\left(
i\right) $, this rewrites as $n-u\left( i\right) \geq n-i$, whence $u(i)<i$.
This proves part \textbf{(b)}.
\end{proof}

\begin{lem}
\label{lem.LRM.1} Let $\leq_{\lex}$ denote the lexicographic order
relation on permutations seen as their one-line notation. Let $w,u\in S_{n}$
be such that $\Des (u)\subseteq \LRM^{\prime }(w)$. Then, $%
uw\;\leq_{\lex}\;w$ .
\end{lem}

\begin{proof}
Assume that $u$ has at least one descent (since otherwise, we have $u\left(
1\right) \leq u\left( 2\right) \leq \cdots \leq u\left( n\right) $,
therefore $u=\id$ and thus $uw=w$, rendering the lemma obvious).
Let $m$ be the largest descent of $u$. Thus, $m\in \Des (u)$. Let $\ell
=m+1$. Then, Lemma \ref{lem.lastdes} \textbf{(a)} yields $u(\ell )<\ell $.
Furthermore, $\ell -1=m\in \Des (u)\subseteq \LRM^{\prime }(w)$,
hence $\ell \in \LRM(w)$. Therefore, every entry $i$ appearing to the
left of $\ell $ in the one-line notation of $w$ is $>\ell $, and thus
satisfies $u(i)\leq i$ by Lemma \ref{lem.lastdes} \textbf{(b)}. Combining
this with $u(\ell )<\ell $, we see that $uw<_{\lex}w$
lexicographically. This proves Lemma \ref{lem.LRM.1}.
\end{proof}

\begin{corollary}
\label{cor.LRM.triang}
For every $w\in S_n$, we have
\[
\mathbf{B}_{\operatorname{LRM}'(w)}\, w
= w + \Bigl(\text{a $\kk$-linear combination of lexicographically smaller permutations}\Bigr).
\]
\end{corollary}

\begin{proof}
Expand:
\[
\mathbf{B}_{\operatorname{LRM}'(w)}
 = \sum_{\substack{u\in S_n;\\\Des(u)\subseteq \operatorname{LRM}'(w)}} u,
\qquad
\mathbf{B}_{\operatorname{LRM}'(w)}\,w
 = \sum_{\substack{u\in S_n;\\\Des(u)\subseteq \operatorname{LRM}'(w)}} uw.
\]
By Lemma~\ref{lem.LRM.1}, every summand $uw$ is lexicographically smaller than $w$, and equality occurs only for $u=\mathrm{id}$.  Thus the claim follows.
\end{proof}

\begin{corollary}
\label{cor.LRM.basis}
The family
\[
\bigl\{
\mathbf{B}_{\operatorname{LRM}'(w)}\, w
\;\big|\; w\in S_n
\bigr\}
\]
is a basis of $\mathcal{A}$.
\end{corollary}

\begin{proof}
Corollary~\ref{cor.LRM.triang} shows that this family expands unitriangularly in the standard basis $\{w\}_{w\in S_n}$.
\end{proof}

\begin{example}
Let us write permutations $\sigma\in S_{n}$ in one-line notation, i.e., as
$\left[  \sigma\left(  1\right)  \ \sigma\left(  2\right)  \ \cdots
\ \sigma\left(  n\right)  \right]  $. For $n=3$, the basis $\bigl\{\mathbf{B}%
_{\operatorname{LRM}^{\prime}(w)}\,w\;\big|\;w\in S_{n}\bigr\}$ of
$\mathcal{A}$ consists of the vectors%
\begin{align*}
\mathbf{B}_{\operatorname{LRM}^{\prime}(\left[  123\right]  )}\,\left[
123\right]    & =\mathbf{B}_{\left(  3\right)  }\ \left[  123\right]  =\left[
123\right]  ;\\
\mathbf{B}_{\operatorname{LRM}^{\prime}(\left[  132\right]  )}\,\left[
132\right]    & =\mathbf{B}_{\left(  3\right)  }\ \left[  132\right]  =\left[
132\right]  ;\\
\mathbf{B}_{\operatorname{LRM}^{\prime}(\left[  213\right]  )}\,\left[
213\right]    & =\mathbf{B}_{\left(  1,2\right)  }\ \left[  213\right]
=\left[  213\right]  +\left[  123\right]  +\left[  132\right]  ;\\
\mathbf{B}_{\operatorname{LRM}^{\prime}(\left[  231\right]  )}\,\left[
231\right]    & =\mathbf{B}_{\left(  1,2\right)  }\ \left[  231\right]
=\left[  231\right]  +\left[  132\right]  +\left[  123\right]  ;\\
\mathbf{B}_{\operatorname{LRM}^{\prime}(\left[  312\right]  )}\,\left[
312\right]    & =\mathbf{B}_{\left(  2,1\right)  }\ \left[  312\right]
=\left[  312\right]  +\left[  213\right]  +\left[  123\right]  ;\\
\mathbf{B}_{\operatorname{LRM}^{\prime}(\left[  321\right]  )}\,\left[
321\right]    & =\mathbf{B}_{\left(  1,1,1\right)  }\ \left[  321\right]
= \ive{321} + \ive{312} + \ive{231} + \ive{213} + \ive{132} + \ive{123}.
\end{align*}
\end{example}

\subsection{Connection with the $\mathbf{B}$-basis filtration}
\label{sec.rs}
We now relate this new basis to the filtration introduced in Section~\ref{sec.filtration}.  Recall the filtration subspaces $\mathcal{R}_\alpha$ indexed by compositions.

\begin{theorem}
\label{thm.LRM.filtbasis}
Let $\alpha\in\Comp_n$ and define
\[
\mathcal{S}_\alpha
 := \operatorname{span}\bigl\{
\mathbf{B}_{\operatorname{LRM}'(w)}\, w
 \;\big|\;
\widetilde{\operatorname{cLRM}'(w)}\preceq_\pi \widetilde{\alpha}
\bigr\}.
\]
Then
\[
\mathcal{S}_\alpha = \mathcal{R}_\alpha.
\]
\end{theorem}

The inclusion
$\mathcal{S}_\alpha \subseteq \mathcal{R}_\alpha$
in Theorem \ref{thm.LRM.filtbasis}
follows directly from Proposition~\ref{cor.Rbeta.subpar}.
The converse inclusion
$\mathcal{R}_\alpha \subseteq \mathcal{S}_\alpha$
is much harder, and will be proved in \S \ref{sec.dynkin}.

\begin{remark}
The use of $\preceq_\pi$ (rather than $\preceq$)
is essential.  
The seemingly more natural span
\(
\mathcal{S}'_\alpha
 :=
 \operatorname{span}\bigl\{
\mathbf{B}_{\operatorname{LRM}'(w)}\, w
\;\big|\;
\operatorname{cLRM}'(w)\preceq \alpha
\bigr\}
\)
is often smaller than $\mathcal{R}_\alpha$.
For instance,
\[
\text{for $n = 3$:} \qquad
\begin{tabular}{|c||c|c|c|c|}
\hline
$\alpha$ & $(1,1,1)$ & $(1,2)$ & $(2,1)$ & $(3)$ \\
\hline
$\dim\bigl(\mathcal{R}_\alpha\bigr)$ & $1$ & $4$ & $4$ & $6$\\
\hline
$\dim\bigl(\mathcal{S}'_\alpha\bigr)$ & $1$ & $3$ & $2$ & $6$\\
\hline
\end{tabular}
\]
\end{remark}

Combining Theorem~\ref{thm.Rbeta.act-on-quot} with
Theorem~\ref{thm.LRM.filtbasis}, we obtain:

\begin{corollary}
\label{cor.evals}
(This is an outline; see Section \ref{sec.evals} for a more
detailed statement.)

Let $\mathbf{a} = \sum_{\alpha \in \Comp_n} \lambda_\alpha \BB_\alpha$
(with $\lambda_\alpha \in \kk$) be an arbitrary element of the
descent algebra $\Sigma_n$.
Then, left multiplication by $\mathbf{a}$ on $\mathcal{A}$
(that is, the linear map $\mathcal{A}\to \mathcal{A},\ \mathbf{x}
\mapsto \mathbf{a}\mathbf{x}$) is represented with respect to the
basis from Corollary~\ref{cor.LRM.basis} by a triangular matrix
(triangular with respect to an appropriate order on $\Comp_n$).
The diagonal entries of this matrix (and thus the eigenvalues
of this map, listed with their algebraic multiplicities) are the
numbers
$\sum_{\alpha \in \Comp_n} \lambda_\alpha \eta_{\cLRM'(w)}\tup{\alpha}$,
where $w$ ranges over $S_n$, and where we use the
$\eta_\beta\tup{\alpha}$ defined in
Theorem~\ref{thm.Rbeta.act-on-quot}.
\end{corollary}

The mere \textbf{existence} of such a basis when $\kk$ is a field
is not new, and is implicit in the works of Bidigare, Brown and
others. The formula for the eigenvalues can
be pieced together from \cite[Proposition 4.2]{RSW2014} and
\cite[Corollary 4.1.3]{Bidigare-thesis} and Foata's fundamental
transformation. Our innovation is
constructing such a basis explicitly, for all commutative rings $\kk$;
this yields a wholly new proof of Corollary~\ref{cor.evals}.

\section{The free algebra and the Dynkin elements}
\label{sec.dynkin}

\subsection{Definitions}


We now introduce the main ingredients in the proof of
Theorem~\ref{thm.LRM.filtbasis}.

Let $W_{n}$ be the monoid of all \emph{words} (i.e., finite tuples) with
letters (i.e., entries) in $\left[  n\right]  $; its multiplication is
concatenation. We denote its generators -- i.e., the letters -- by $\underline{1}%
,\underline{2},\ldots,\underline{n}$, so as to distinguish them from the
respective numbers. The monoid algebra $F_n := \kk \left[  W_{n}\right]  $ of
$W_{n}$ is the ring of \emph{noncommutative polynomials} (aka the \emph{free
algebra}) in $n$ indeterminates
$\underline{1},\underline{2},\ldots,\underline{n}$ over $\kk$.
It is graded by degree (i.e., word length),
and we let $F_{n,k}$ denote its $k$-th graded component, with basis given by
the set $W_{n,k}$ of length-$k$ words. The symmetric group $S_{k}$ acts on
$W_{n,k}$ from the right by permuting the letters:%
\[
\left(  \underline{w_{1}}\ \underline{w_{2}}\ \cdots\ \underline{w_{k}%
}\right)  \cdot\sigma=\underline{w_{\sigma\left(  1\right)  }}%
\ \underline{w_{\sigma\left(  2\right)  }}\ \cdots\ \underline{w_{\sigma
\left(  k\right)  }}\qquad\text{for }w_{i}\in\left[  n\right]  \text{ and
}\sigma\in S_{k}.
\]
Extending this action linearly, we obtain a right $\kk \left[
S_{k}\right]  $-action on $\kk \left[  W_{n,k}\right] = F_{n,k}$.
Two words in
$W_{n,k}$ are said to be \emph{anagrams} (of each other) if they belong to the
same $S_{k}$-orbit.

There is an injective map $S_{n}\rightarrow W_{n,n}$, sending each permutation
$\sigma\in S_{n}$ to its one-line notation $\underline{\sigma}%
:=\underline{\sigma\left(  1\right)  }\ \underline{\sigma\left(  2\right)
}\ \cdots\ \underline{\sigma\left(  n\right)  }$. By linearity, it extends to
an injective $\kk$-linear map $\kk \left[  S_{n}\right]
\rightarrow \kk \left[  W_{n,n}\right]  = F_{n,n}$, which we also write as
$a\mapsto\underline{a}$. This map is not a ring morphism, but it respects the
right $\kk \left[  S_{n}\right]  $-actions, i.e., satisfies%
\begin{align}
\underline{a}\cdot b=\underline{ab}\qquad\text{for each }a,b\in
\kk \left[  S_{n}\right]  \text{.}
\label{eq.OLN-equivariant}
\end{align}

A word $w=\left(  w_{1},w_{2},\ldots,w_{k}\right)  $ is said to have
\emph{V-shape} if it satisfies
\[
w_{1}>w_{2}>\cdots>w_{i}<w_{i+1}<w_{i+2}<\cdots<w_k
\qquad \text{ for some } i \in \ive{k}
\]
(that is, it first decreases, then increases). In this case, the $i$ is unique
and is denoted $\operatorname*{val}w$. A permutation $v\in S_n$ is said to
have \emph{V-shape} on an interval $[i, j] =
\left\{  i,i+1,\ldots,j\right\}  $ if the
word $\tup{ v\tup{i}, v\tup{i+1}, \ldots, v\tup{j} }$ has V-shape.

When $a,b\in F_{n}$, we write $\left[  a,b\right]  $ for the \emph{commutator}
$ab-ba$.

The crucial notion to us is the following:

\begin{defn}
\label{def.dynkin}For any nonempty subset $S$ of $\left[  n\right]  $, we
define the \emph{Dynkin element} (or \emph{nested commutator})
\[
V^{S}:=\left[  \left[  \cdots\left[  \left[  \underline{s_{1}}%
,\ \underline{s_{2}}\right]  ,\ \underline{s_{3}}\right]  ,\ \ldots\right]
,\ \underline{s_{k}}\right]  ,
\]
where $s_{1},s_{2},\ldots,s_{k}$ are the elements of $S$ in increasing order.
(If $k=1$, then this means just $\underline{s_{1}}$.)
\end{defn}

\begin{example}
We have
\begin{align*}
V^{\left\{  2,6,7\right\}  }  & =\left[  \left[  \underline{2},\ \underline{6}%
\right]  ,\ \underline{7}\right]  =\underline{2}\ \underline{6}\ \underline{7}%
-\underline{6}\ \underline{2}\ \underline{7}-\underline{7}\ \underline{2}%
\ \underline{6}+\underline{7}\ \underline{6}\ \underline{2};\\
V^{\left\{  1,4\right\}  }  & =\left[  \underline{1},\ \underline{4}\right]
=\underline{1}\ \underline{4}-\underline{4}\ \underline{1}%
;\ \ \ \ \ \ \ \ \ \ V^{\left\{  3\right\}  }=\underline{3}.
\end{align*}

\end{example}

The letter $V$ is chosen because of the following simple formula:

\begin{proposition}
\label{prop.dynkin.Vsh}
Let $S$ be a nonempty subset of $\left[  n\right]  $.
Then,
\[
V^{S}=\sum\left(  -1\right)  ^{\operatorname*{val}w-1}w,
\]
where the sum ranges over all words $w$ that contain each element of $S$ exactly
once (and no other letters) and have V-shape.
\end{proposition}

\begin{proof}
Easy induction on $\left\vert S\right\vert $, since $V^{S}=\left[
V^{S\setminus\left\{  \max S\right\}  },\ \underline{\max S}\right]  $ for
$\left\vert S\right\vert >1$.
\end{proof}

Dynkin elements are fundamental to the study of the free Lie algebra; see,
e.g., \cite[(8.4.2)]{Reutenauer1993}. We shall have more use for their
products in $F_n$.

\begin{defn}
\label{defn.valpha} Let $\alpha=(\alpha_{1},\alpha_2,\dots,\alpha_{p})\in
\operatorname{Comp}(n)$ be a composition of $n$. Then, we set
\[
\mathbf{V}_{\alpha}=V^{\operatorname{Set}(\alpha)_{1}}V^{\operatorname{Set}%
(\alpha)_{2}}\cdots V^{\operatorname{Set}(\alpha)_{p}}.
\]

\end{defn}

\begin{example}
For the compositions $(2,1)$ and $(2,2)$ we obtain:
\begin{align*}
\mathbf{V}_{(2,1)}
&=V^{\left\{  1,2\right\}  }V^{\left\{  3\right\}  }=\left(
\underline{1}\ \underline{2}-\underline{2}\ \underline{1}\right)
\ \underline{3}=\underline{1}\ \underline{2}\ \underline{3}-\underline{2}%
\ \underline{1}\ \underline{3};
\\
\mathbf{V}_{(2,2)}  & =V^{\left\{  1,2\right\}  }V^{\left\{  3,4\right\}
}=\left(  \underline{1}\ \underline{2}-\underline{2}\ \underline{1}\right)
\left(  \underline{3}\ \underline{4}-\underline{4}\ \underline{3}\right)
=\underline{1}\ \underline{2}\ \underline{3}\ \underline{4}-\underline{1}%
\ \underline{2}\ \underline{4}\ \underline{3}-\underline{2}\ \underline{1}%
\ \underline{3}\ \underline{4}+\underline{2}\ \underline{1}\ \underline{4}%
\ \underline{3}.
\end{align*}

\end{example}


\subsection{Action on the $\mathbf{B}$-basis}


The Dynkin elements $\mathbf{V}_{\alpha}$ are themselves images of certain
elements of the descent algebra $\Sigma_{n}$ under the injection
$\kk \left[  S_{n}\right]  \rightarrow \kk \left[  W_{n,n}\right]
= F_{n,n}
,\ a\mapsto\underline{a}$, namely of the \textquotedblleft\emph{power sums of
the first kind}\textquotedblright\ $\Psi_{\alpha}$ studied in \cite[\S 5]%
{NCSF1}. However, we will use another of their properties
-- namely, a formula
for the action of $\mathbf{B}_{\beta}\in \kk \left[  S_{n}\right]  $ on
the $\mathbf{V}_{\alpha}$, which appears as Lemma 9.33 in
\cite{Reutenauer1993} and as Theorem 2.1 in \cite{GarReu89}:

\begin{theorem}[Lemma 9.33 of \cite{Reutenauer1993}, or Theorem 2.1 of \cite{GarReu89}]\label{thm.vb}
Let $\beta = \tup{\beta_1,\beta_2,\ldots,\beta_p}$
and $\gamma = \tup{\gamma_1,\gamma_2,\ldots,\gamma_q}$
be two compositions of $n$.
Let $\mathcal{L}_{\beta,\gamma}$ be the set of all ordered set partitions
$T=(T_{1},T_{2},\ldots,T_{q})$ of the index set
$[p]=\left\{  1,2,\ldots ,p\right\}  $ into $q$ blocks that satisfy
\begin{align}
\sum_{i\in T_{j}}\beta_{i} = \gamma_{j} \quad\text{for all }j.
\label{eq.thm.vb.ass}
\end{align}
For each block $T_j$ of
$T$, let $V^{!T_{j}}$ denote the product
of the Dynkin elements $V^{\Set(\beta)_{i}}$ with $i\in T_{j}$
(in the order of increasing $i$). Then
\begin{align}
\mathbf{V}_{\beta}\,\mathbf{B}_{\gamma}=\sum_{T=\left(  T_{1},T_{2}%
,\ldots,T_{q}\right)
\in \mathcal{L}_{\beta,\gamma}} V^{!T_{1}} V^{!T_2} \cdots V^{!T_{q}}.
\label{eq.thm.vb.clm}
\end{align}

\end{theorem}

\begin{remark}
Reutenauer's \cite[Lemma 9.33]{Reutenauer1993}
is even more general: He views $F_n$ as the tensor algebra $T\tup{V}$
of the free $\kk$-module $V = \kk^n$ with basis $\tup{\underline 1,
\underline 2,\ldots,\underline n}$.
This equips $F_n$ with a Hopf algebra structure: that of the tensor
Hopf algebra $T\tup{V}$.
All Dynkin elements $V^S$ are primitive elements of $T\tup{V}$,
as can be easily proved by induction on $\abs{S}$.

Let $\gamma = \left(
\gamma_{1},\gamma_{2},\ldots,\gamma_{q} \right) $ be a composition of $n$, and
let $P_{1}, P_{2}, \ldots, P_{k}$ be $k$ homogeneous primitive
elements of $T\tup{V}$ with total degree
$\deg P_{1} + \deg P_{2} + \cdots+ \deg P_{k} = n$.

Set $P_{W} := P_{w_{1}} P_{w_{2}} \cdots P_{w_{m}}$ for any subset $W =
\left\{  w_{1} < w_{2} < \cdots< w_{m} \right\} $ of $[k]$.

Then, the right action of the symmetric group $S_{n}$ on $F_{n,n}$
satisfies
\[
\left(  P_{1} P_{2} \cdots P_{k} \right)  \mathbf{B}_{\gamma}
= \sum P_{T_{1}} P_{T_{2}} \cdots P_{T_{q}},
\]
where the sum ranges over all ordered set partitions $\left(  T_{1}, T_{2},
\ldots, T_{q} \right) $ of $[k]$ such that $\deg P_{T_{i}} = \gamma_{i}$ for
each $i$.

This can be proved slickly using Hopf algebra theory; here is an outline
of the proof:
\begin{enumerate}
\item First, observe that the action
of $\mathbf{B}_{\gamma}$ on the $n$-th degree component of $T\left(  V\right)
$ is the Hopf-algebraic operator $m^{[q]}\circ P_{\gamma}\circ\Delta^{[q]}$
(where $\Delta^{[q]} : T\tup{V} \to T\tup{V}^{\otimes q}$
is the $q$-fold comultiplication,
$m^{[q]} : T\tup{V}^{\otimes q} \to T\tup{V}$ is the
$q$-fold multiplication, and $P_\gamma$ is the projection
$T\tup{V}^{\otimes q} \to T\tup{V}^{\otimes q}$ onto the
$\tup{ \gamma_{1},\gamma_{2},\ldots,\gamma_{q} }$-th multigraded component),
since $\Delta^{[q]}$ sends each ``word'' $v_1 v_2 \cdots v_n \in T\tup{V}$
(with $v_1, v_2, \ldots, v_n \in V$) to the tensor
$\sum_{W_1 \sqcup W_2 \sqcup \cdots \sqcup W_q = [n]}
v_{W_1} \otimes v_{W_2} \otimes \cdots \otimes v_{W_q} \in T\tup{V}^{\otimes q}$,
where $v_W$ for a subset $W =
\left\{  w_{1} < w_{2} < \cdots< w_{m} \right\} \subseteq [n]$
is defined to be $v_{w_1} v_{w_2} \cdots v_{w_m} \in T\tup{V}$.
\item Then, show that the latter
operator $m^{[q]}\circ P_{\gamma}\circ\Delta^{[q]}$
sends $P_{1}P_{2}\cdots P_{k}$ to $\sum P_{T_{1}}P_{T_{2}}\cdots
P_{T_{q}}$.
The easiest way to do so is by observing that this holds when
$P_1, P_2, \ldots, P_k$ are themselves among the generators of
the free algebra $T\tup{V}$ (that is, belong to a basis of $V$),
but the general case then follows easily
(since the primitivity of $P_1, P_2, \ldots, P_k$ ensures that
there is a bialgebra morphism from a free algebra to $T\tup{V}$
that sends its generators to $P_{1},P_{2},\ldots,P_{k}$).
\end{enumerate}
\end{remark}


\begin{remark}
The products $V^{!T_1} V^{!T_2} \cdots V^{!T_{q}}$ in Theorem~\ref{thm.vb}
are -- in general -- not of the form $\mathbf{V}_{\gamma}$. For instance,
\[
V^{\{3\}}V^{\{1,2\}}\;
=\;\underline3(\underline1\ \underline2-\underline2\ \underline1)\;
=\;\underline3\ \underline1\ \underline2-\underline3\ \underline2\ \underline1,
\]
a summand in $\mathbf{V}_{(2,1)}\mathbf{B}_{(1,2)}$,
differs from both
\[
\mathbf{V}_{(2,1)}\;=\;V^{\{1,2\}}V^{\{3\}}\;
=\;(\underline1\ \underline2-\underline2\ \underline1)\underline3\;
=\;\underline1\ \underline2\ \underline3
-\underline2\ \underline1\ \underline3
\]
and
\[
\mathbf{V}_{(1,2)}\;
=\;\underline1(\underline2\ \underline3-\underline3\ \underline2)\;
=\;\underline1\ \underline2\ \underline3-\underline1\ \underline3\ \underline2.
\]
\end{remark}

\begin{example}

Consider $\alpha=(2,2)$ and $\beta=(3,1)$. We have
\begin{align*}
\mathbf{V}_{(2,2)} &  =(\ul1\ul2-\ul2\ul1)(\ul3\ul4-\ul4\ul3)=\ul1\ul2\ul3\ul4-\ul2\ul1\ul3\ul4-\ul1\ul2\ul4\ul3+\ul2\ul1\ul4\ul3,\\
\mathbf{B}_{(3,1)} &  =1234+1243+1342+2341.
\end{align*}
Hence
\begin{align*}
\mathbf{V}_{(2,2)}\mathbf{B}_{(3,1)} &
=(\ul1\ul2\ul3\ul4-\ul2\ul1\ul3\ul4-\ul1\ul2\ul4\ul3+\ul2\ul1\ul4\ul3)(1234+1243+1342+2341)\\
&=\ul1\ul2\ul3\ul4-\ul2\ul1\ul3\ul4-\ul1\ul2\ul4\ul3+\ul2\ul1\ul4\ul3
+\ul1\ul2\ul4\ul3-\ul2\ul1\ul4\ul3-\ul1\ul2\ul3\ul4+\ul2\ul1\ul3\ul4\\
&\qquad \qquad +\ul1\ul3\ul4\ul2-\ul2\ul3\ul4\ul1-\ul1\ul4\ul3\ul2+\ul2\ul4\ul3\ul1
+\ul2\ul3\ul4\ul1-\ul1\ul3\ul4\ul2-\ul2\ul4\ul3\ul1+\ul1\ul4\ul3\ul2\\
&  =0.
\end{align*}

On the other hand, for $\alpha=\beta=(2,2)$,
\begin{align*}
\mathbf{V}_{(2,2)}\mathbf{B}_{(2,2)}  & = (\ul1\ul2\ul3\ul4 - \ul2\ul1\ul3\ul4 - \ul1\ul2\ul4\ul3 + \ul2\ul1\ul4\ul3) (1234 +
1423 + 1324 + 2413 + 2314 + 3412)\\[2mm]
& =\, V_{2}^{\{1,2\}} V_{2}^{\{3,4\}} + V_{2}^{\{3,4\}} V_{2}^{\{1,2\}} \;\neq\; 0.
\end{align*}

\end{example}


\section{Proof of the main theorem}
\label{sec.proof}

A direct consequence of Theorem~\ref{thm.vb} is the following vanishing
criterion:
\begin{align}
\widetilde{\beta}\npreceq_{\pi}\widetilde{\gamma}
\quad\Longrightarrow\quad
\mathbf{V}_{\beta}\,\mathbf{B}_{\gamma}=0
\label{eq.VB-zero}
\end{align}
(since $\mathcal{L}_{\beta,\gamma} = \varnothing$ in this case).
A slightly less obvious particular case is the following:

\begin{lem}
\label{lem.VB-eqsize}Let $\beta=\left(  \beta_{1},\beta_{2},\ldots,\beta
_{p}\right)  $ and $\gamma=\left(  \gamma_{1},\gamma_{2},\ldots,\gamma
_{p}\right)  $ be two compositions of $n$ having the same length. Then,%
\[
\mathbf{V}_{\beta}\,\mathbf{B}_{\gamma}=\sum_{\substack{\chi\in S_{p}%
;\\\beta_{\chi\left(  i\right)  }=\gamma_{i}\text{ for each }i}%
}
V^{\Set(\beta)_{\chi(1)}} V^{\Set(\beta)_{\chi(2)}} \cdots
V^{\Set(\beta)_{\chi(p)}}.
\]

\end{lem}

\begin{proof}
Set $q = p$, and apply Theorem \ref{thm.vb}.
Observe that any ordered set partition $T \in \mathcal{L}_{\beta, \gamma}$
must consist entirely of $1$-element blocks (since it partitions a
$p$-element set into $p$ blocks), and thus has the form
$\tup{\set{\chi\tup{1}}, \set{\chi\tup{2}}, \ldots,
\set{\chi\tup{p}}}$ for a unique permutation $\chi \in S_p$.
Moreover, this latter permutation $\chi$ must satisfy
$\beta_{\chi\tup{i}} = \gamma_i$ for all $i \in [p]$,
in order for $T$ to satisfy \eqref{eq.thm.vb.ass}.
This is necessary and sufficient for $T \in \mathcal{L}_{\beta, \gamma}$.
Thus, \eqref{eq.thm.vb.clm} rewrites as
\begin{align*}
\mathbf{V}_{\beta}\,\mathbf{B}_{\gamma}  & =\sum_{\substack{\chi\in
S_{p};\\\beta_{\chi\left(  i\right)  }=\gamma_{i}\text{ for each }%
i}}V^{!\left\{  \chi\left(  1\right)  \right\}  }\cdots V^{!\left\{  \chi\left(
p\right)  \right\}  }
=\sum_{\substack{\chi\in S_{p};\\\beta_{\chi\left(  i\right)  }=\gamma
_{i}\text{ for each }i}}V^{\Set(\beta)_{\chi(1)}}\cdots
V^{\Set(\beta)_{\chi(p)}}.
\end{align*}
This proves Lemma \ref{lem.VB-eqsize}.
\end{proof}

We shall now study the products
$V^{\Set(\beta)_{\chi(1)}} V^{\Set(\beta)_{\chi(2)}} \cdots
V^{\Set(\beta)_{\chi(p)}}$
on the right hand side of Lemma \ref{lem.VB-eqsize} more closely.
These products are multilinear in the generators
$\underline 1,\underline 2,\ldots,\underline n$ of $F_n$,
so they
are linear combinations of words
$\underline v$ with $v \in S_n$.
The next lemma reveals some insights about the $v$'s that
appear in these combinations.

\begin{lem}
\label{lem.wuv}Let $w,u\in S_{n}$. Assume that the compositions $\beta=\left(
\beta_{1},\beta_{2},\ldots,\beta_{p}\right)  :=\operatorname{cLRM}^{\prime
}(w)$ and $\gamma=\left(  \gamma_{1},\gamma_{2},\ldots,\gamma_{p}\right)
:=\operatorname{cLRM}^{\prime}(u)$ are anagrams of one another, and let
$\chi\in S_{p}$ be a permutation satisfying%
\begin{equation}
\beta_{\chi\left(  i\right)  }=\gamma_{i}\qquad\text{for each }i\in\left[
p\right]  .
\label{eq.lem.wuv.bg}
\end{equation}

Let $v\in S_{n}$ be a permutation such that the word $\underline{v}\in
W_{n,n}$ appears in the expansion of the product $V^{\operatorname{Set}%
(\beta)_{\chi(1)}}\cdots V^{\operatorname{Set}(\beta)_{\chi(p)}}$.
Equivalently, let
$v\in S_{n}$ be a permutation that sends each $\operatorname*{Set}\left(
\gamma\right)  _{i}$ to $\operatorname*{Set}\left(  \beta\right)
_{\chi\left(  i\right)  }$ and has V-shape on each $\operatorname*{Set}\left(
\gamma\right)  _{i}$
(this is equivalent by Proposition~\ref{prop.dynkin.Vsh} and by
\eqref{eq.lem.wuv.bg}). Assume that
\[
vu=w.
\]
Then, $\chi= \id $ and $v= \id $ and $\gamma=\beta$.
\end{lem}

\begin{proof}
For each $i\in\left[  p\right]  $, we define a claim $\mathcal{C}\tup{i}$
as follows:
\begin{quote}
\textit{Claim $\mathcal{C}\tup{i}$:}
We have $\chi\tup{i} = i$
and $\gamma_i = \beta_i$
and $\Set\tup{\beta}_i = \Set\tup{\gamma}_i$
and $\left. v\mid_{\Set\tup{\beta}_i} \right. = \id$.
\end{quote}
It clearly suffices to show that $\mathcal{C}\left(  i\right)  $ holds for all
$i\in\left[  p\right]  $. We will prove this by strong descending induction
on $i$. That is, we fix $i\in\left[  p\right]  $, and we assume that
all the ``higher'' claims
$\mathcal{C}\left(  i+1\right)  ,\mathcal{C}\left(  i+2\right)  ,\ldots
,\mathcal{C}\left(  p\right)  $ hold. We must then prove $\mathcal{C}\left(
i\right)  $.

By \eqref{eq.alpha-blocks},
we have $\operatorname*{Set}\left(
\beta\right)  _{i}=\left[  \beta_{\leq i-1}+1,\ \beta_{\leq i}\right]  $
and $\operatorname*{Set}\left(  \gamma\right)  _{i}=\left[
\gamma_{\leq i-1}+1,\ \gamma_{\leq i}\right]  $.

If $\sigma\in S_{n}$ is a permutation and $j\leq\left\vert \operatorname*{LRM}%
\sigma\right\vert $ is a positive integer, then we let $\operatorname*{lrm}%
\nolimits_{j}\sigma$ denote the $j$-th largest LRM of $\sigma$. We extend this
to $j=0$ by setting $\operatorname*{lrm}\nolimits_{0}\sigma:=n+1$. We note
that the LRMs of a permutation $\sigma$ appear in decreasing order when
reading $\sigma$ from left to right, so that $\operatorname*{lrm}%
\nolimits_{j}\sigma$ is the $j$-th LRM of $\sigma$ encountered in such a read.
Thus, it is easy to see that each $\sigma\in S_{n}$ and each positive
$j\leq\left\vert \operatorname*{LRM}\sigma\right\vert $ satisfy%
\begin{equation}
\operatorname*{lrm}\nolimits_{j}\sigma=\sigma\left(  \min\left(  \sigma
^{-1}\left(  \left[  \operatorname*{lrm}\nolimits_{j-1}\sigma-1\right]
\right)  \right)  \right)  .\label{pf.lem.wuv.lrm-min}%
\end{equation}
(Indeed, this is just saying that if we remove all numbers larger than
$\operatorname*{lrm}\nolimits_{j-1}\sigma-1$ from the one-line notation of
$\sigma$, then the first of the remaining numbers will be $\operatorname*{lrm}%
\nolimits_{j}\sigma$; but this is clear from the definition of an LRM.)

But $\beta=\operatorname{cLRM}^{\prime}\left(  w\right)  $ shows that
$\beta_{\leq i-1}+1$ is the $i$-th smallest LRM of $w$, and hence is the
$\left(  p+1-i\right)  $-th largest LRM of $w$ (since $w$ has $\ell\left(
\operatorname{cLRM}^{\prime}(w)\right)  =\ell\tup{\beta} = p$
many LRMs in total). That is,
$\beta_{\leq i-1}+1=\operatorname*{lrm}\nolimits_{p+1-i}w$. Likewise,
$\beta_{\leq i}+1=\operatorname*{lrm}\nolimits_{p-i}w$. But
(\ref{pf.lem.wuv.lrm-min}) (applied to $\sigma=w$ and $j=p+1-i$) yields
\[
\operatorname*{lrm}\nolimits_{p+1-i}w=w\left(  \min\left(  w^{-1}\left(
\left[  \operatorname*{lrm}\nolimits_{p-i}w-1\right]  \right)  \right)
\right)  .
\]
Using $\beta_{\leq i-1}+1=\operatorname*{lrm}\nolimits_{p+1-i}w$ and
$\beta_{\leq i}+1=\operatorname*{lrm}\nolimits_{p-i}w$, this rewrites as%
\begin{equation}
\beta_{\leq i-1}+1=w\left(  \min\left(  w^{-1}\left(  \left[  \beta_{\leq
i}\right]  \right)  \right)  \right)  .\label{pf.lem.wuv.4b}%
\end{equation}
Similarly,%
\begin{equation}
\gamma_{\leq i-1}+1=u\left(  \min\left(  u^{-1}\left(  \left[  \gamma_{\leq
i}\right]  \right)  \right)  \right)  .\label{pf.lem.wuv.4g}%
\end{equation}

However, our induction hypothesis yields $\beta_{j}=\gamma_{j}$ for all $j>i$,
and thus easily $\beta_{\leq i}=\gamma_{\leq i}$ (since $\left\vert
\beta\right\vert =n=\left\vert \gamma\right\vert $). Furthermore, it shows
that $\left. v\mid_{\operatorname*{Set}\left(  \beta\right)  _{j}} \right.
= \id $ for all $j>i$, and thus
$\left. v\mid_{\left[  n\right]  -\left[  \beta_{\leq
i}\right]  }\right. = \id $,
so that
$v\tup{\left[  n\right]  -\left[  \beta_{\leq
i}\right]} = \left[  n\right]  -\left[  \beta_{\leq
i}\right]$.
By taking complements, we obtain
$v\tup{\left[  \beta_{\leq i}\right]} = \left[  \beta_{\leq i}\right]$,
thus
$v^{-1}\tup{\left[  \beta_{\leq i}\right]} = \left[  \beta_{\leq i}\right]$.
But $vu = w$, so that $w^{-1}=u^{-1}v^{-1}$ and therefore
$w^{-1}\left( \left[  \beta_{\leq i}\right]  \right)
= u^{-1}\left( v^{-1} \left( \left[  \beta_{\leq i}\right]  \right)\right)
= u^{-1}\left(  \left[  \beta_{\leq i}\right]  \right)  $,
because
$v^{-1}\tup{\left[  \beta_{\leq i}\right]} = \left[  \beta_{\leq i}\right]$.

Now, applying $v$ to (\ref{pf.lem.wuv.4g}), we find
\begin{align}
v\left(  \gamma_{\leq i-1}+1\right)    & =v\left(  u\left(  \min\left(
u^{-1}\left(  \left[  \gamma_{\leq i}\right]  \right)  \right)  \right)
\right)  \nonumber\\
& =w\left(  \min\left(  u^{-1}\left(  \left[  \beta_{\leq i}\right]  \right)
\right)  \right)  \ \ \ \ \ \ \ \ \ \ \left(  \text{since }vu=w
\text{ and } \gamma_{\leq i} = \beta_{\leq i} \right)
\nonumber\\
& =w\left(  \min\left(  w^{-1}\left(  \left[  \beta_{\leq i}\right]  \right)
\right)  \right)  \ \ \ \ \ \ \ \ \ \ \left(  \text{since }w^{-1}\left(
\left[  \beta_{\leq i}\right]  \right)  =u^{-1}\left(  \left[  \beta_{\leq
i}\right]  \right)  \right)  \nonumber\\
& =\beta_{\leq i-1}+1\ \ \ \ \ \ \ \ \ \ \left(  \text{by (\ref{pf.lem.wuv.4b}%
)}\right)  \label{pf.lem.wuv.v1}\\
& \in\left[  \beta_{\leq i-1}+1,\ \beta_{\leq i}\right]  =\operatorname*{Set}%
\left(  \beta\right)  _{i}.\label{pf.lem.wuv.vSet}%
\end{align}
But $\gamma_{\leq i-1}+1\in\left[  \gamma_{\leq i-1}+1,\ \gamma_{\leq
i}\right]  =\operatorname*{Set}\left(  \gamma\right)  _{i}$, so that%
\[
v\left(  \gamma_{\leq i-1}+1\right)  \in v\left(  \operatorname*{Set}\left(
\gamma\right)  _{i}\right)  =\operatorname*{Set}\left(  \beta\right)
_{\chi\left(  i\right)  }%
\]
(by the requirements on $v$). Confronting this with (\ref{pf.lem.wuv.vSet}),
we conclude that $\chi\left(  i\right)  =i$, since otherwise
$\operatorname*{Set}\left(  \beta\right)  _{\chi\left(  i\right)  }$ would be
disjoint from $\operatorname*{Set}\left(  \beta\right)  _{i}$. Thus,
(\ref{eq.lem.wuv.bg}) becomes $\gamma_{i}=\beta_{i}$. Thus, $\gamma_{\leq
i-1}=\beta_{\leq i-1}$ (by subtracting $\gamma_{i}=\beta_{i}$ from
$\gamma_{\leq i}=\beta_{\leq i}$). Combining this with $\gamma_{\leq i}%
=\beta_{\leq i}$, we conclude that $\operatorname*{Set}\left(  \beta\right)
_{i}=\operatorname*{Set}\left(  \gamma\right)  _{i}$.

Recall that the permutation $v$ has V-shape on $\operatorname*{Set}\left(
\gamma\right)  _{i} =\left[
\gamma_{\leq i-1}+1,\ \gamma_{\leq i}\right] $. That is, the word
$\tup{v\tup{\gamma_{\leq i-1} + 1}, v\tup{\gamma_{\leq i-1} + 2}, \ldots,
v\tup{\gamma_{\leq i}}}$ has V-shape.
However, the letters of this word are $\gamma_{\leq i-1} + 1,
\gamma_{\leq i-1} + 2, \ldots, \gamma_{\leq i}$ in some order
(since $v$ sends $\operatorname*{Set}\left(
\gamma\right)  _{i}$ to $\operatorname*{Set}\left(  \beta\right)
_{\chi\left(  i\right)  }=\operatorname*{Set}\left(  \beta\right)
_{i}=\operatorname*{Set}\left(  \gamma\right)  _{i}$),
and moreover the \textbf{leftmost} of these letters is the
\textbf{smallest} of them (since $\gamma_{\leq
i-1}=\beta_{\leq i-1}$ allows us to rewrite (\ref{pf.lem.wuv.v1}) as $v\left(
\gamma_{\leq i-1}+1\right)  =\gamma_{\leq i-1}+1$).
But a word that has V-shape and starts with its smallest letter
must be increasing; thus, in particular, the word
$\tup{v\tup{\gamma_{\leq i-1} + 1}, v\tup{\gamma_{\leq i-1} + 2}, \ldots,
v\tup{\gamma_{\leq i}}}$ is increasing.
Since its letters are $\gamma_{\leq i-1} + 1,
\gamma_{\leq i-1} + 2, \ldots, \gamma_{\leq i}$, we thus obtain
\[
\tup{v\tup{\gamma_{\leq i-1} + 1}, v\tup{\gamma_{\leq i-1} + 2}, \ldots,
v\tup{\gamma_{\leq i}}}
= \tup{\gamma_{\leq i-1} + 1,
\gamma_{\leq i-1} + 2, \ldots, \gamma_{\leq i}}.
\]
In other words,
$\left. v\mid_{\operatorname*{Set}\left(  \gamma\right)  _{i}}\right.
= \id $.
That is, $\left. v\mid_{\operatorname*{Set}\left(  \beta\right)  _{i}}%
\right. = \id $ (since $\operatorname*{Set}\left(  \beta\right)
_{i}=\operatorname*{Set}\left(  \gamma\right)  _{i}$).
We have now proved all four parts of our claim $\mathcal{C}\left(  i\right)
$. Thus, the induction step is complete, and as we said, this finishes the
proof of the lemma.
\end{proof}

For any $a \in F_n$ and any $\mathfrak w \in W_{n}$, let
$[\mathfrak w](a)$ denote the $\mathfrak w$-coefficient of $a$.

\begin{lem}
\label{lem.wu2}Let $w,u\in S_{n}$. Set $\beta:=\operatorname{cLRM}^{\prime
}(w)$ and $\gamma:=\operatorname{cLRM}^{\prime}(u)$. Then:

\textbf{(a)} We have $\left[ \underline w\right]
\left(  \mathbf{V}_{\beta} \mathbf{B}_{\beta}w\right)  =1$.

\textbf{(b)} If $u\neq w$, then $\left[  \underline w\right]
\left(  \mathbf{V}_{\beta
}\mathbf{B}_{\gamma}u\right)  =0$ unless $\widetilde{\beta}\prec_{\pi
}\widetilde{\gamma}$.
\end{lem}

\begin{proof}
\textbf{(a)} Write $\beta$ as $\beta=\left(  \beta_{1},\beta_{2},\ldots
,\beta_{p}\right)  $.
Applying Lemma \ref{lem.VB-eqsize} to $\gamma=\beta$, we
find%
\[
\mathbf{V}_{\beta}\,\mathbf{B}_{\beta}=\sum_{\substack{\chi\in S_{p}%
;\\\beta_{\chi\left(  i\right)  }=\beta_{i}\text{ for each }i}}
V^{\operatorname{Set}(\beta)_{\chi(1)}}\cdots
V^{\operatorname{Set}(\beta)_{\chi(p)}}.
\]
Hence,%
\begin{equation}
\left[  \underline w \right]  \left(  \mathbf{V}_{\beta}\mathbf{B}_{\beta}w\right)
=\sum_{\substack{\chi\in S_{p};\\\beta_{\chi\left(  i\right)  }=\beta
_{i}\text{ for each }i}}\left[  \underline w\right]  \left( 
V^{\operatorname{Set}(\beta)_{\chi(1)}}\cdots
V^{\operatorname{Set}(\beta)_{\chi(p)}}\cdot w\right)  .\label{pf.lem.wu2.5}%
\end{equation}
However, if $\chi\in S_{p}$ satisfies $\beta_{\chi\left(  i\right)  }=\beta_{i}$
for each $i$, then
\begin{equation}
\left[  \underline w\right]  \left(  V^{\operatorname{Set}(\beta
)_{\chi(1)}}\cdots V^{\operatorname{Set}(\beta)_{\chi(p)}%
}w\right)  =%
\begin{cases}
1, & \text{if }\chi= \id ;\\
0, & \text{if }\chi\neq \id ,
\end{cases}
\label{pf.lem.wu2.6}%
\end{equation}
because Lemma \ref{lem.wuv} (applied to $u=w$ and $\gamma=\beta$)
shows that a term $\underline v$ in the expansion of the
product
$V^{\Set(\beta)_{\chi(1)}}\cdots V^{\Set(\beta)_{\chi(p)}}$
(where $v \in S_n$; keep in mind that this product belongs
to the image of $\kk[S_n]$ under $a \mapsto \underline a$)
can satisfy $vw = w$ (that is, $\underline vw=\underline w$)
only if $\chi= \id $ and
$v= \id $, and it is easy to see
(using Proposition~\ref{prop.dynkin.Vsh}) that the coefficient of
$\underline{ \id }$ in $V^{\operatorname{Set}(\beta
)_{\chi(1)}}\cdots V^{\operatorname{Set}(\beta)_{\chi(p)}}$
is $1$ for $\chi= \id $.
Using (\ref{pf.lem.wu2.6}), we can rewrite (\ref{pf.lem.wu2.5}) as%
\[
\left[  \underline w\right]  \left(  \mathbf{V}_{\beta}\mathbf{B}_{\beta}w\right)
=\sum_{\substack{\chi\in S_{p};\\\beta_{\chi\left(  i\right)  }=\beta
_{i}\text{ for each }i}}%
\begin{cases}
1, & \text{if }\chi= \id ;\\
0, & \text{if }\chi\neq \id %
\end{cases}
\ \ =1.
\]
This proves part \textbf{(a)}.
\medskip

\textbf{(b)} Assume that $u\neq w$, but we don't have $\widetilde{\beta}%
\prec_{\pi}\widetilde{\gamma}$. We must show that $[\underline w]\left(  \mathbf{V}%
_{\beta}\mathbf{B}_{\gamma}u\right)  =0$. If we don't have $\widetilde{\beta
}\preceq_{\pi}\widetilde{\gamma}$, then this is immediate from
(\ref{eq.VB-zero}). So we WLOG assume that we do have $\widetilde{\beta}\preceq_{\pi
}\widetilde{\gamma}$. Hence, $\widetilde{\beta}=\widetilde{\gamma}$ (since we
don't have  $\widetilde{\beta}\prec_{\pi}\widetilde{\gamma}$). Thus, the
compositions $\beta$ and $\gamma$ are anagrams of one another, and in
particular, have the same length; write them as $\beta=\left(  \beta_{1}%
,\beta_{2},\ldots,\beta_{p}\right)  $ and $\gamma=\left(  \gamma_{1}%
,\gamma_{2},\ldots,\gamma_{p}\right)  $. Now, Lemma \ref{lem.VB-eqsize} yields%
\begin{align}
\left[  \underline w\right]  \left(  \mathbf{V}_{\beta}\mathbf{B}_{\gamma}u\right)    &
=\left[  \underline w\right]  \left(  \sum_{\substack{\chi\in S_{p};\\\beta_{\chi\left(
i\right)  }=\gamma_{i}\text{ for each }i}}
V^{\operatorname{Set}(\beta)_{\chi(1)}}\cdots
V^{\operatorname{Set}(\beta)_{\chi(p)}}\cdot u\right)  \nonumber\\
& =\sum_{\substack{\chi\in S_{p};\\\beta_{\chi\left(  i\right)  }=\gamma
_{i}\text{ for each }i}}\left[  \underline w\right]  \left( 
V^{\operatorname{Set}(\beta)_{\chi(1)}}\cdots 
V^{\operatorname{Set}(\beta)_{\chi(p)}}\cdot u\right)  .
\label{pf.lem.wu2.8}%
\end{align}
But each addend $\left[ \underline w\right]  \left( 
V^{\operatorname{Set}(\beta)_{\chi(1)}}\cdots
V^{\operatorname{Set}(\beta)_{\chi(p)}}\cdot u\right)  $ in this sum is $0$,
since Lemma \ref{lem.wuv} shows that a term $\underline v$ in the expansion of
$V^{\operatorname{Set}(\beta)_{\chi(1)}}\cdots
V^{\operatorname{Set}(\beta)_{\chi(p)}}$ can satisfy $vu=w$
(that is, $\underline v u = \underline w$)
only if
$\chi= \id $ and $v= \id $, but this is impossible
(since $v= \id $ would simplify $vu=w$ to $u=w$, contradicting
our assumption $u\neq w$). Hence, the whole sum in (\ref{pf.lem.wu2.8}) is
$0$, and we obtain $\left[  \underline w\right]  \left(  \mathbf{V}_{\beta}%
\mathbf{B}_{\gamma}u\right)  =0$. This proves part \textbf{(b)}.
\end{proof}

\begin{proof}
[Proof of Theorem \ref{thm.LRM.filtbasis}.] We showed that $\mathcal{S}%
_{\alpha}\subseteq\mathcal{R}_{\alpha}$ in Section \ref{sec.rs}. Thus, it
remains to prove that $\mathcal{R}_{\alpha}\subseteq\mathcal{S}_{\alpha}$. So
let $b\in\mathcal{R}_{\alpha}$ be arbitrary. We must show that $b\in
\mathcal{S}_{\alpha}$.
The family $\{\mathbf{B}_{\cLRM'%
(w)}\,w\}_{w\in S_{n}}=\{\mathbf{B}_{\operatorname*{LRM}\nolimits^{\prime}%
(w)}\,w\}_{w\in S_{n}}$ is a basis of $\mathcal{A}$ by Corollary
\ref{cor.LRM.basis}. Hence, we can expand $b$ as%
\begin{equation}
b=\sum_{w\in S_{n}}\lambda_{w}\,\mathbf{B}_{\operatorname*{cLRM}%
\nolimits^{\prime}(w)}\,w\label{pf.thm.LRM.filtbasis.exp}%
\end{equation}
with $\lambda_{w}\in \kk$.
We claim that $\lambda_{w}=0$ for all $w\in S_{n}$ satisfying
$\widetilde{\operatorname{cLRM}^{\prime}(w)}\npreceq_{\pi}\widetilde{\alpha}$.
Once this claim is proved, (\ref{pf.thm.LRM.filtbasis.exp}) will reduce to a
sum over the $w$ that satisfy $\widetilde{\operatorname{cLRM}^{\prime}(w)}%
\preceq_{\pi}\widetilde{\alpha}$, thus showing that $b\in\mathcal{S}_{\alpha}$
and thus concluding our proof.
So we must only show our claim that $\lambda_{w}=0$ for all $w\in S_{n}$
satisfying $\widetilde{\operatorname{cLRM}^{\prime}(w)}\npreceq_{\pi
}\widetilde{\alpha}$. We prove this by strong induction on
$\widetilde{\operatorname{cLRM}^{\prime}(w)}$ (with respect to the reversed
partition refinement order). That is, we fix a partition $\mu\vdash n$ with
$\mu\npreceq_{\pi}\widetilde{\alpha}$, and we assume as induction hypothesis
that $\lambda_{w}=0$ is already proved for all $w\in S_{n}$ satisfying
$\mu\prec_{\pi}\widetilde{\operatorname{cLRM}^{\prime}(w)}$ (note that
$\mu\prec_{\pi}\widetilde{\operatorname{cLRM}^{\prime}(w)}$ and $\mu
\npreceq_{\pi}\widetilde{\alpha}$ automatically entail
$\widetilde{\operatorname{cLRM}^{\prime}(w)}\npreceq_{\pi}\widetilde{\alpha}$,
so we need not additionally require the latter). We then must prove that
$\lambda_{w}=0$ for all $w\in S_{n}$ satisfying
$\widetilde{\operatorname{cLRM}^{\prime}(w)}=\mu$.

Let $w\in S_{n}$ satisfy $\widetilde{\operatorname{cLRM}^{\prime}(w)}=\mu$.
Set $\beta:=\operatorname{cLRM}^{\prime}(w)$, so that $\widetilde{\beta
}=\widetilde{\operatorname{cLRM}^{\prime}(w)}=\mu\npreceq_{\pi}%
\widetilde{\alpha}$. Hence, $\mathbf{V}%
_{\beta}\mathbf{B}_{\alpha}=0$ by (\ref{eq.VB-zero}).
Thus, $\mathbf{V}_{\beta}b=0$ as well
(because $b\in\mathcal{R}_{\alpha}=\mathbf{B}_{\alpha}\mathcal{A}$). Hence,%
\begin{align*}
0  & =\mathbf{V}_{\beta}b=\mathbf{V}_{\beta}\cdot\sum_{u\in S_{n}}\lambda
_{u}\,\mathbf{B}_{\cLRM'(u)}\,u\qquad\left(
\text{by (\ref{pf.thm.LRM.filtbasis.exp})}\right)  \\
& =\sum_{u\in S_{n}}\lambda_{u}\,\mathbf{V}_{\beta}\mathbf{B}%
_{\cLRM'(u)}\,u.
\end{align*}
Taking the $\underline w$-coefficient, we thus conclude that%
\begin{equation}
0=\left[  \underline w\right]  \left(  \sum_{u\in S_{n}}\lambda_{u}\,\mathbf{V}_{\beta
}\mathbf{B}_{\cLRM'(u)}\,u\right)  =\sum_{u\in
S_{n}}\lambda_{u}\,\left[  \underline w\right]  \left(  \mathbf{V}_{\beta}\mathbf{B}%
_{\cLRM'(u)}\,u\right)
.\label{pf.thm.LRM.filtbasis.0=}%
\end{equation}
But the sum on the right hand side here has many vanishing addends:
\begin{enumerate}
\item All the addends with $\mu\prec_{\pi}\widetilde{\operatorname{cLRM}%
^{\prime}(u)}$ are $0$ (since the induction hypothesis says that $\lambda
_{u}=0$ for all such $u$).
\item All the addends that satisfy $u\neq w$ but not $\mu\prec_{\pi
}\widetilde{\operatorname{cLRM}^{\prime}(u)}$ are $0$ (since Lemma
\ref{lem.wu2} \textbf{(b)}, applied to $\gamma=\operatorname*{cLRM}%
\nolimits^{\prime}\left(  u\right)  $,
yields $\left[  \underline w\right]  \left(
\mathbf{V}_{\beta}\mathbf{B}_{\cLRM'%
(u)}\,u\right)  =0$).
\end{enumerate}
The only remaining addend is the one for $u=w$; this addend equals
$\lambda_{w}$ (since we have $\left[  \underline w\right]  \left(  \mathbf{V}_{\beta}%
\mathbf{B}_{\cLRM'(w)}\,w\right)  =\left[
\underline w\right]  \left(  \mathbf{V}_{\beta}\mathbf{B}_{\beta}w\right)  =1$ by Lemma
\ref{lem.wu2} \textbf{(a)}). Thus, the whole sum simplifies to $\lambda_w$.
Consequently, (\ref{pf.thm.LRM.filtbasis.0=}) rewrites as $0=\lambda_{w}$,
thus $\lambda_{w}=0$, as desired. This completes the induction step, hence
establishing our claim.
\end{proof}

\section{Eigenvalues}
\label{sec.evals}

We shall now elaborate upon and prove Corollary~\ref{cor.evals}. We begin with
notations for left and right multiplications on $\mathcal{A}$:

\begin{defn}
Let $\mathbf{a}\in\mathcal{A}$. Then, we let $L\left(  \mathbf{a}\right)
:\mathcal{A}\rightarrow\mathcal{A}$ and $R\left(  \mathbf{a}\right)
:\mathcal{A}\rightarrow\mathcal{A}$ be the maps that send each $\mathbf{x}%
\in\mathcal{A}$ to $\mathbf{ax}$ and to $\mathbf{xa}$, respectively. These
maps $L\left(  \mathbf{a}\right)  $ and $R\left(  \mathbf{a}\right)  $ are
$\mathbf{k}$-linear, and are known respectively as \emph{left }and \emph{right
multiplication} by $\mathbf{a}$.
\end{defn}

As linear endomorphisms of the free $\mathbf{k}$-module $\mathcal{A}%
=\mathbf{k}\left[  S_{n}\right]  $, these maps $L\left(  \mathbf{a}\right)  $
and $R\left(  \mathbf{a}\right)  $ can be represented by square matrices with
respect to any basis of $\mathcal{A}$, and these matrices have eigenvalues,
eigenvectors and characteristic polynomials. In general, their eigenvalues can
fail to belong to $\mathbf{k}$; but not so when $\mathbf{a}$ belongs to the
descent algebra $\Sigma_{n}$. Namely:

\begin{corollary}
\label{cor.evals2}The \emph{LRM-basis} shall mean the basis $\bigl\{\mathbf{B}%
_{\LRM'(w)}\,w\;\big|\;w\in S_{n}\bigr\}$ of $\mathcal{A}$
from
Corollary~\ref{cor.LRM.basis}. We shall equip its indexing set $S_{n}$ with a
(strict) total order $\vartriangleleft$, in which two permutations $u,v\in
S_{n}$ satisfy $u\vartriangleleft v$ if $\widetilde{\operatorname*{cLRM}%
\nolimits^{\prime}\left(  u\right)  }\prec_{\pi}%
\widetilde{\cLRM'\left(  v\right)  }$. (This
is, per se, only a partial order, but we pick any linear extension.)

Let $\mathbf{a}=\sum_{\alpha\in\operatorname{Comp}_{n}}\lambda_{\alpha
}\mathbf{B}_{\alpha}$ (with $\lambda_{\alpha}\in\mathbf{k}$) be an arbitrary
element of the descent algebra $\Sigma_{n}$. Then:

\begin{enumerate}
\item[\textbf{(a)}] The map $L\left(  \mathbf{a}\right)  $ is represented by
an upper-triangular matrix with respect to the LRM-basis.

\item[\textbf{(b)}] The eigenvalues of the map $L\left(  \mathbf{a}\right)  $
(listed with their algebraic multiplicities) are the numbers $\sum_{\alpha
\in\operatorname{Comp}_{n}}\lambda_{\alpha}\eta_{\operatorname{cLRM}^{\prime
}(w)}\left(  \alpha\right)  $, where $w$ ranges over $S_{n}$, and where we use
the $\eta_{\beta}\left(  \alpha\right)  $ defined in
Theorem~\ref{thm.Rbeta.act-on-quot}.
\end{enumerate}
\end{corollary}

\begin{proof}
\textbf{(a)} The map $L\left(  \mathbf{a}\right)  $ depends $\mathbf{k}%
$-linearly on $\mathbf{a}$. Thus, from $\mathbf{a}=\sum_{\alpha\in
\operatorname{Comp}_{n}}\lambda_{\alpha}\mathbf{B}_{\alpha}$, we obtain
$L\left(  \mathbf{a}\right)  =\sum_{\alpha\in\operatorname{Comp}_{n}}%
\lambda_{\alpha}L\left(  \mathbf{B}_{\alpha}\right)  $. Hence, it suffices to
show that each of the $L\left(  \mathbf{B}_{\alpha}\right)  $ is represented
by an upper-triangular matrix with respect to the LRM-basis. In other words,
it suffices to show that for each $\alpha\in\Comp_{n}$
and $v\in S_{n}$, we have%
\begin{align}
& \left(  L\left(  \mathbf{B}_{\alpha}\right)  \right)  \left(  \mathbf{B}%
_{\LRM'(v)}\,v\right)  \nonumber\\
& =\left(  \text{a linear combination of }\mathbf{B}_{\operatorname{LRM}%
^{\prime}(u)}\,u\text{ with }u\trianglelefteq v\right)
\label{pf.cor.evals2.a.goal}%
\end{align}
(where \textquotedblleft$\trianglelefteq$\textquotedblright\ means
\textquotedblleft$\vartriangleleft$ or $=$\textquotedblright). Let $\alpha
\in\Comp_{n}$ and $v\in S_{n}$. Set $\beta
=\cLRM'\left(  v\right)  $. Then,
$\mathbf{B}_{\beta}=\mathbf{B}_{\LRM'(v)}$.
Meanwhile,
\begin{align}
\left(  L\left(  \mathbf{B}_{\alpha}\right)  \right)  \left(
\mathbf{B}_{\beta}\,v\right)
&  =\mathbf{B}_{\alpha}\ \mathbf{B}_{\beta}\,v\qquad\left(  \text{by the
definition of }L\left(  \mathbf{B}_{\alpha}\right)  \right)  \nonumber\\
&  =\left(  \eta_{\beta}\left(  \alpha\right)  \ \mathbf{B}_{\beta}+\left(
\text{a sum of }\mathbf{B}_{\alpha^{\prime}}\text{ with }\alpha^{\prime}%
\prec\beta\right)  \right)  v\qquad\left(  \text{by \eqref{eq.BaBb=sum-prec}}%
\right)  \nonumber\\
&  =\eta_{\beta}\left(  \alpha\right)  \ \mathbf{B}_{\beta}v+\left(  \text{a
sum of }\mathbf{B}_{\alpha^{\prime}}v\text{ with }\alpha^{\prime}\prec
\beta\right) .\nonumber
\end{align}
In view of%
\begin{align*}
&  \left(  \text{a sum of }\mathbf{B}_{\alpha^{\prime}}v\text{ with }%
\alpha^{\prime}\prec\beta\right)  \\
&  \in\sum_{\substack{\alpha^{\prime}\in\operatorname{Comp}_{n};\\\alpha
^{\prime}\prec\beta}}\mathcal{R}_{\alpha^{\prime}}\qquad\left(  \text{since
}\mathbf{B}_{\alpha^{\prime}}v\in\mathbf{B}_{\alpha^{\prime}}\mathcal{A}%
=\mathcal{R}_{\alpha^{\prime}}\right)  \\
&  \subseteq\sum_{\substack{\alpha^{\prime}\in\operatorname{Comp}%
_{n};\\\widetilde{\alpha^{\prime}}\prec_{\pi}\widetilde{\beta}}}\mathcal{R}%
_{\alpha^{\prime}}\qquad\left(  \text{because }\alpha^{\prime}\prec\beta\text{
entails }\widetilde{\alpha^{\prime}}\prec_{\pi}\widetilde{\beta}\right)  \\
&  =\sum_{\substack{\alpha^{\prime}\in\operatorname{Comp}_{n}%
;\\\widetilde{\alpha^{\prime}}\prec_{\pi}\widetilde{\beta}}%
}\operatorname*{span}\left\{  \mathbf{B}_{\operatorname{LRM}^{\prime}%
(u)}\,u\text{ }\mid\text{ }\widetilde{\operatorname{cLRM}^{\prime}\left(
u\right)  }\preceq_{\pi}\widetilde{\alpha^{\prime}}\right\}  \\
&  \qquad\qquad\left(  \text{by Theorem \ref{thm.LRM.filtbasis}, summed over
all $\alpha^{\prime}$ with $\widetilde{\alpha^{\prime}}\prec_{\pi
}\widetilde{\beta}$}\right)  \\
&  =\operatorname*{span}\left\{  \mathbf{B}_{\operatorname{LRM}^{\prime}%
(u)}\,u\text{ }\mid\text{ }\widetilde{\operatorname{cLRM}^{\prime}\left(
u\right)  }\prec_{\pi}\widetilde{\beta}\right\}  \\
&  \qquad\qquad\left(  \text{since }\widetilde{\operatorname{cLRM}^{\prime
}\left(  u\right)  }\prec_{\pi}\widetilde{\beta}\text{ if and only if
}\widetilde{\operatorname{cLRM}^{\prime}\left(  u\right)  }\preceq_{\pi
}\widetilde{\alpha^{\prime}}\text{ for some }\widetilde{\alpha^{\prime}}%
\prec_{\pi}\widetilde{\beta}\right)  \\
&  \subseteq\operatorname*{span}\left\{  \mathbf{B}_{\operatorname{LRM}%
^{\prime}(u)}\,u\text{ }\mid\text{ }u\vartriangleleft v\right\}  \\
&  \qquad\qquad\left(  \text{since }\widetilde{\operatorname{cLRM}^{\prime
}\left(  u\right)  }\prec_{\pi}\widetilde{\beta}%
=\widetilde{\operatorname*{cLRM}\nolimits^{\prime}\left(  v\right)  }\text{
entails }u\vartriangleleft v\right)  ,
\end{align*}
this leads to%
\begin{equation}
\left(  L\left(  \mathbf{B}_{\alpha}\right)  \right)  \left(  \mathbf{B}%
_{\beta}\,v\right)  \in\eta_{\beta}\left(  \alpha\right)  \ \mathbf{B}_{\beta
}\ v+\operatorname*{span}\left\{  \mathbf{B}_{\operatorname{LRM}^{\prime}%
(u)}\,u\text{ }\mid\text{ }u\vartriangleleft v\right\}  .\nonumber
\end{equation}
In view of
$\mathbf{B}_{\beta}=\mathbf{B}_{\LRM'(v)}$ and
$\beta=\cLRM'\left(  v\right)  $, this
rewrites as%
\begin{align}
& \left(  L\left(  \mathbf{B}_{\alpha}\right)  \right)  \left(  \mathbf{B}%
_{\LRM'(v)}\,v\right)  \nonumber\\
& \in\eta_{\cLRM'\left(  v\right)  }\left(
\alpha\right)  \ \mathbf{B}_{\LRM'(v)}%
v+\operatorname*{span}\left\{  \mathbf{B}_{\LRM'%
(u)}\,u\text{ }\mid\text{ }u\vartriangleleft v\right\}
\label{pf.cor.evals2.a.4}\\
& \subseteq\operatorname*{span}\left\{  \mathbf{B}_{\operatorname{LRM}%
^{\prime}(u)}\,u\text{ }\mid\text{ }u\trianglelefteq v\right\}  ,\nonumber
\end{align}
which is precisely (\ref{pf.cor.evals2.a.goal}). Thus, Corollary
\ref{cor.evals2} \textbf{(a)} is proved.
\medskip

\textbf{(b)} By part \textbf{(a)}, the map $L\left(  \mathbf{a}\right)  $ is
represented by an upper-triangular matrix with respect to the LRM-basis.
Hence, the eigenvalues of $L\left(  \mathbf{a}\right)  $ are the diagonal
entries of this matrix. But $L\left(  \mathbf{a}\right)  =\sum_{\alpha
\in\operatorname{Comp}_{n}}\lambda_{\alpha}L\left(  \mathbf{B}_{\alpha
}\right)  $ (as we saw in the proof of part \textbf{(a)}). Hence, for each
$v\in S_{n}$, we have%
\begin{align*}
& \left(  L\left(  \mathbf{a}\right)  \right)  \left(  \mathbf{B}%
_{\LRM'(v)}\,v\right)  \\
& =\sum_{\alpha\in\operatorname{Comp}_{n}}\lambda_{\alpha}\underbrace{\left(
L\left(  \mathbf{B}_{\alpha}\right)  \right)  \left(  \mathbf{B}%
_{\LRM'(v)}\,v\right)  }_{\substack{\in\eta
_{\cLRM'\left(  v\right)  }\tup{\alpha}\ \mathbf{B}%
_{\LRM'(v)}v+\operatorname*{span}\left\{  \mathbf{B}%
_{\LRM'(u)}\,u\text{ }\mid\text{ }u\vartriangleleft
v\right\}  \\\text{(by (\ref{pf.cor.evals2.a.4}))}}}\\
& \in\sum_{\alpha\in\operatorname{Comp}_{n}}\lambda_{\alpha}\left(
\eta_{\cLRM'\left(  v\right)  }\left(
\alpha\right)  \ \mathbf{B}_{\LRM'(v)}%
v+\operatorname*{span}\left\{  \mathbf{B}_{\LRM'%
(u)}\,u\text{ }\mid\text{ }u\vartriangleleft v\right\}  \right)  \\
& \subseteq\left(  \sum_{\alpha\in\operatorname{Comp}_{n}}
\lambda_{\alpha} \eta_{\cLRM'\left(  v\right)  }\tup{\alpha}\right)
\mathbf{B}_{\LRM'(v)}v+\operatorname*{span}\left\{
\mathbf{B}_{\LRM'(u)}\,u\text{ }\mid\text{
}u\vartriangleleft v\right\}  .
\end{align*}
In other words, the diagonal entries of the matrix that represents $L\left(
\mathbf{a}\right)  $ with respect to the LRM-basis are the scalars
$\sum_{\alpha\in\operatorname{Comp}_{n}}\lambda_{\alpha}\eta
_{\operatorname{cLRM}^{\prime}(v)}\left(  \alpha\right)  $ for all $v\in
S_{n}$. So these scalars are the eigenvalues of $L\left(  \mathbf{a}\right)
$. This proves Corollary \ref{cor.evals2} \textbf{(b)}.
\end{proof}

Part \textbf{(b)} (but not part \textbf{(a)}) of Corollary \ref{cor.evals2}
holds also for $R\left(  \mathbf{a}\right)  $ instead of $L\left(
\mathbf{a}\right)  $. This follows from Corollary
\ref{cor.evals2} \textbf{(b)} via
the following fact that has nothing to do with $\Sigma_{n}$:

\begin{proposition}
\label{prop.L=R}Let $\mathbf{a}\in\mathcal{A}$ be arbitrary. Then:

\begin{enumerate}
\item[\textbf{(a)}] The linear maps $L\left(  \mathbf{a}\right)  $ and
$R\left(  \mathbf{a}\right)  $ have the same characteristic polynomial, thus
the same eigenvalues (with the same algebraic multiplicities).

\item[\textbf{(b)}] Let $X$ be the matrix representing the map $L\left(
\mathbf{a}\right)  $ with respect to the basis $\left(  w\right)  _{w\in
S_{n}}$ of $\mathcal{A}$. Let $Y$ be the matrix representing the map $R\left(
\mathbf{a}\right)  $ with respect to the basis $\left(  w^{-1}\right)  _{w\in
S_{n}}$ of $\mathcal{A}$. Then, $X=Y^{T}$.
\end{enumerate}
\end{proposition}

\begin{proof}
\textbf{(b)} By linearity, it suffices to prove this when $\mathbf{a}$ is a
single permutation $p\in S_{n}$. In this case, the claim boils down to
\textquotedblleft$pu=v$ if and only if $v^{-1}p=u^{-1}$\textquotedblright. But
the latter is a trivial calculation.
\medskip

\textbf{(a)} This follows from part \textbf{(b)}, since mutually transpose
matrices have the same characteristic polynomial.
\end{proof}

\section{Questions}
\label{sec.quests}

Having proved the main results, we propose two questions.

\begin{enumerate}

\item[\textbf{1.}]
On the basis $\bigl\{
\mathbf{B}_{\operatorname{LRM}'(w)}\, w
\;\big|\; w\in S_n \bigr\}$
from Corollary~\ref{cor.LRM.basis}, left multiplication
by any element of $\Sigma_n$ acts as a triangular matrix.
However, right multiplication generally does not.
Is there a basis that represents both left and right
multiplication by triangular matrices?
Theoretical reasoning using the structure of $\Sigma_n$
shows that the answer is ``yes''\footnote{Indeed, if
$\kk$ is a field of characteristic $0$, then the
$\kk$-algebra $\Sigma_n$ becomes a direct product of
$\kk$'s after we quotient it by its Jacobson radical
(\cite[\S 5.2]{Sch05}).
Thus, the same is true of the algebra
$\Sigma_n \otimes \Sigma_n^{\operatorname{op}}$.
Hence, each
left $\Sigma_n \otimes \Sigma_n^{\operatorname{op}}$-module
over a field of characteristic $0$ has a basis on which
$\Sigma_n \otimes \Sigma_n^{\operatorname{op}}$ acts by
triangular matrices.
In other words, each
\((\Sigma_n,\Sigma_n)\)-bimodule
over a field of characteristic $0$ has a basis on which
\(\Sigma_n\) acts by triangular matrices from both left
and right.
This applies in particular to the
\((\Sigma_n,\Sigma_n)\)-bimodule \(\mathcal{A}\).
Using the structure of finitely generated modules over
a PID, this can be
extended to $\kk = \ZZ$, and thus to all $\kk$.}, but does not
suggest a concrete basis.\footnote{NB: The
analogue of Lemma~\ref{prop.Rbeta.sym} for the left
ideals $\mathcal{L}_\alpha := \mathcal{A} \BB_\alpha$
instead of the right ideals $\mathcal{R}_\alpha$ is false.}

\item[\textbf{2.}]
A descent algebra can be defined for any finite Coxeter group \cite{Sol76}.
It is desirable to know if an analogue of LRMs, and a basis such
as the one in Corollary~\ref{cor.LRM.basis}, still exists in this generality.
Note that \cite[\S 3]{ChenGongGuo13} defines an analogue of LRMs in type $B$.

\end{enumerate}

\printbibliography

\end{document}